\documentclass[12pt]{article}
\usepackage{amsfonts,amssymb,amsthm,amsmath,latexsym,bm,enumerate,authblk,color}
\usepackage[colorlinks,linktocpage,linkcolor=blue]{hyperref}

\allowdisplaybreaks

\numberwithin{equation}{section}
\newtheorem{theorem}{Theorem}[section]
\newtheorem{lemma}[theorem]{Lemma}

\newtheorem{assumption}{Assumption}[section]
\newtheorem{problem}{Problem}

\def\d{\mathrm d}

\def\N{\mathbb N}
\def\C{\mathbb C}
\def\al{\alpha}

\def\te{\theta}
\def\na{\nabla}

\def\Om{\Omega}
\def\pa{\partial}
\def\wh{\widehat}
\def\LL{{L^2(\Om)}}

\def\Om{\Omega}
\def\al{\alpha}

\def\te{\theta}

\def\vp{\varphi}

\def\nb{\nabla}

\def\pa{\partial}
\def\wh{\widehat}

\title{Uniqueness and stability for inverse source problem for fractional diffusion-wave equations }
\author[1]{Xing Cheng\thanks{chengx@hhu.edu.cn}}
\author[2]{Zhiyuan Li\thanks{Corresponding author: zyli@sdut.edu.cn}}
\affil[1]{\normalsize{College of Science, Hohai University, Nanjing 210098, Jiangsu, China}}
\affil[2]{\normalsize{School of Mathematics and Statistics, Shandong University of Technology, Zibo 255049, Shandong, China}}

\date{}

\begin{document}
\maketitle

\begin{abstract}
This paper is devoted to the inverse problem of determining the spatially dependent source in a time fractional diffusion-wave equation, with the aid of extra measurement data at subboundary. Uniqueness result is obtained by using the analyticity and the new established unique continuation principle provided that the coefficients are all temporally independent. We also derive a Lipschitz stability of our inverse source problem under a suitable topology whose norm is given via the adjoint system of the fractional diffusion-wave equation.

\medskip


\medskip

\end{abstract}

\section{Introduction and main results}\label{sec-intro}
We consider the following time-fractional diffusion-wave equation
\begin{align}\label{eq-gov}
\begin{cases}
\partial_t^\alpha u(x,t) +A(x)u(x,t) = F(x,t) &\text{ in } \Omega \times (0, T), \\
u(x,0 ) =a(x),\quad u_t(x,0 ) = b(x)& \text{ in } \Omega,\\
u(x,t )=0 & \text{ on }\partial \Omega \times (0, T),
\end{cases}
\end{align}
where $\Omega$ is a subset of $\mathbb R^d$ with a sufficiently smooth boundary $\partial\Omega$ and $(0,T)$ is a time interval.
$\partial_t^\alpha$ is the Caputo derivative of order $\alpha \in (1,2)$ defined by
\begin{align*}
\partial_t^\alpha g (t) = \frac1{ \Gamma ( 2 - \alpha)} \int_0^t (t - \tau)^{1 - \alpha } g''(\tau) \,\mathrm{d}\tau,  \quad t > 0,
\end{align*}
for $g\in AC^1[0,T]:= \left\{g,g'\in AC[0,T] \right\}$. $A(x)$ is an elliptic operator which is defined by
$$
A(x) u =  -\sum_{i,j=1}^d \partial_{x_j} \left(a_{ij}(x) \partial_{x_j}u \right) + B(x)\cdot \nabla u + c(x) u, \quad \forall \, u\in H_0^1(\Omega) \cap H^2(\Omega),
$$
where $a_{ij}\in C^1 \left(\overline\Omega \right)$ with $a_{ij}=a_{ji}$, and there exists a positive constant $a_0$ such that
\begin{equation}
\label{condi-elliptic}
\sum_{i,j=1}^d a_{ij}(x) \xi_i \xi_j \ge a_0 \sum_{j=1}^d |\xi_j|^2,\quad \forall\,  x\in\overline\Omega,\ \forall\,  (\xi_1,\cdots,\xi_d)\in \mathbb R^d.
\end{equation}
The governing equation in \eqref{eq-gov} is a time-fractional diffusion-wave equation of super-diffusion type.
It has gained more and more attentions from many scientists which is due to many reasons, for example, fractional operators take into account, not only the present and local states of phenomenon, but also historical ones, see e.g.,
\cite{Butzer, HLW, Herrmann, Main96, Main10, Metzler, Pov08, Pov15, SW89} for
various applications of fractional diffusion-wave equations in aerodynamics, biology, physics, chemistry, quantum mechanics, nuclear physics,
hadron spectroscopy and so on. Recently, there is also a lot of literatures concerning
\eqref{eq-gov} from both theoretical and numerical aspects, we refer to
\cite{Agrawal02, Go00, JiangLiu, Kem12, Ko14, Lu18, Pon20, Pov12, SY11} and the references therein.

In this paper, we will focus on the inverse problem of \eqref{eq-gov} in the following formulation:
\begin{problem}[Inverse source problem]\label{st1v2}
 Let $F(x,t)=f(x) g(t)$ with $g\in C^1[0,T]$ and $f\in H_0^1(\Omega)$ in \eqref{eq-gov}.
 Let $\Gamma$ be a nonempty and open subset of the boundary $\partial\Omega$. The following problems arise naturally.
\begin{enumerate}[1)]
\item (Uniqueness)
 Given $\phi(x,t) = \partial_{\nu_A}u(x, t):=\sum\limits_{i,j=1}^d a_{ij}(x) \partial_{x_i}u(x,t) \nu_j$, $(x,t)\in\Gamma\times(0,T)$, where $\nu:=(\nu_1,\cdots,\nu_d)$ denotes the outward unit normal vector on the boundary $\partial\Omega$, as boundary measurement for our inverse problem. Does the boundary data $\phi(x,t)$ uniquely determine the source term $f$?
\item (Stability)
Since the measurement is noise-contaminated inevitability, how to establish a stability estimate for the correspondence  $\phi \mapsto f$?
\end{enumerate}
\end{problem}
Inverse problem is one of the most important fields in the applied mathematics, which consists of recoverying, e.g., the source term, initial state of the given system by means of additional measurements based on forward problems.
To the authors' best knowledge, the study on inverse problems for fractional differential equations is focused on $0 < \alpha < 1$
(see, e.g. Jiang, Li, Liu and Yamamoto \cite{JLLY17}), Yamamoto and Zhang \cite{YZ12}, Zhang and Xu \cite{ZX11}. For $1 < \alpha < 2$, there are few theoretical results on the inverse source problems except for Hu, Liu and Yamamoto \cite{HLY18, LHY20}, Liao and Wei \cite{LiaoWei} and Yan and Wei \cite{YanWei20}.
In \cite{LiaoWei} the fractional order and the source term in a time-fractional diffusion-wave equation were simultaneously identified. \cite{YanWei20} is devoted to identify a space-dependent source term in a symmetric time fractional diffusion-wave equation from a part of noisy boundary data and the uniqueness result which is established via Titchmarsh convolution theorem and the Duhamel principle.
 \cite{HLY18, LHY20} are concerned with the determination of moving source profile functions arising from time-fractional diffusion-wave equations, where the unique identification of two source profiles with distinct moving directions was verified by the unique continuation principle.

For identifying other parameters, such as fractional orders, initial values, coefficients, we can refer to Floridia and Yamamoto \cite{FY20}, Yan and Wei \cite{YanWei20}, Yan, Zhang and Wei \cite{YanZhangWei}. We also refer to Jin and Rundell \cite{RunJin15}, Li, Liu and Yamamoto \cite{LiLY19}, Li and Yamamoto \cite{LY19} and Liu, Li and Yamamoto \cite{LiuLY19} for a topical review and a comprehensive list of bibliographies. For numerical treatments on inverse problems for fractional diffusion-wave equations, we refer to Wei and Zhang \cite{WeiZhang}, Xian and Wei \cite{XianWei} and the references therein.

By now,  almost all existing literatures treating the related inverse problems heavily rely on the symmetry of the fractional differential equations.  As for inverse source problem for nonsymmetric fractional diffusion-wave equations they are still at an early stage. Thus, in this paper, we will deal with an inverse source problem for nonsymmetric diffusion-wave equation 

{The inverse problem in Problem \ref{st1v2}} resolves the dilemma that many important quantities cannot be observed directly,
while it costs less to observe the considered system in a subboundary $\Gamma$ than to observe it in the whole domain $\Omega$. We will now give the answer to the Problem \ref{st1v2}. Before presenting our results, we introduce the Sobolev space $H_\alpha(0,T)$ of order $\alpha\in(1,2)$:
\begin{align}\label{eq2.2v2}
H_\alpha(0,T):= \left\{g\in H^1(0,T); \frac{d}{dt} g\in H_{\alpha-1}(0,T) \right\}
\end{align}
with the norm
$$
\|g\|_{H_\alpha(0,T)}:= \|g\|_{L^2(0,T)} + \left\|\frac{d}{dt} g \right\|_{H_{\alpha-1}(0,T)}.
$$
This kind of Sobolev space is a subspace of the usual fractional Sobolev space $H^\alpha(0,T)$, we refer to \cite{GLY15}.
\begin{theorem}\label{thm-unique}
Assume $F(x,t)=f(x)g(t)$ with $g(\cdot)$ known as a $C^1[0, T]$-function such that $g(0) \ne 0$, and let $u\in H_\alpha\left(0,T;L^2(\Omega) \right) \cap L^2 \left(0,T; H_0^1(\Omega) \cap H^2(\Omega) \right)$ be the solution to the problem \eqref{eq-gov} with $f\in H_0^1(\Omega)$. Then $\partial_{\nu_A} u(x,t) = 0$, $(x,t)\in\Gamma \times (0, T)$ implies $f = 0$ in $\Omega$.
\end{theorem}
We can further derive a Lipschitz type stability of our inverse source problem under a suitable topology.
\begin{theorem}\label{thm-stabi}
Under the same assumptions in Theorem \ref{thm-unique}, furthermore suppose $(u_j,f_j)$, $j=1,2$ are two pairs of solutions of our inverse source problem
which correspond to the measurement data $\partial_{\nu_A} u_j(x,t)$, $j=1,2$. Then we have
$$
\|f_1-f_2\|_{\mathcal B} \le C \left\|\partial_{\nu_A} u_1 - \partial_{\nu_A} u_2 \right\|_{L^2(\Gamma\times(0,T))},
$$
where the norm $\mathcal{B}$ is the norm associated with the bilinear form defined in \eqref{eq2.3v3}.
\end{theorem}

We list the main difficulties and the innovations of this paper. The essential difficulties are
\begin{enumerate}
\item the nonsymmetic term makes the usual analytical technique used in Sakamoto and Yamamoto \cite{SY11}, like Mittag-Leffler expression and Laplace transform argument, cannot work directly; Indeed, it is difficult to give an exact formula for the solution of our inverse source problem.
\item the classical Caputo derivative $\partial_t^\alpha g$ needs $g\in AC^1[0,T]$, and this space is too narrow to find one solution for several important applications. Moreover, the fractional derivative is nonlocal, which makes that there are quite few mathematical tools to derive the stability result.
\end{enumerate}
The main innovations of this paper include
\begin{enumerate}
\item we understand the Caputo derivative in the operator sense, and prove the wellposedness for the forward problem by regarding the non-symmetric term as a new source in the framework of fractional Sobolev spaces;
\item we obtain the uniqueness for the inverse source problem on the basis of a new established unique continuation for the fractional diffusion-wave equation in Lemma \ref{lem-ucp};
\item we equip $L^2(\Omega)$ with a new norm which is constructed by the adjoint system of the fractional diffusion-wave, from which the continuous dependence of solution of problem \ref{st1v2}  with respect to the additional measurement data is
also verified.
\end{enumerate}

\section{Preliminaries}\label{sec-pre}
This section is devoted to some basic definitions and properties,
covering both the fractional calculus and adjoint system, which will be used throughout this paper.

\subsection{Fractional derivatives}
We start by giving some definitions of fractional derivatives and basic function spaces. The Riemann-Liouville fractional integral $J^\gamma$ of order $\gamma>0$ is defined to be
\begin{equation}
\label{defi-RL}
J^\gamma g(t) = \frac1{\Gamma(\gamma)} \int_0^t (t-\tau)^{\gamma-1} g(\tau) d\tau,\quad t>0,
\end{equation}
where $\Gamma(\cdot)$ is the Gamma function. In the case of $\gamma\in(0,1)$, 
it is known that the Riemann-Liouville integral operator $J^{\gamma}: L^2(0,T) \to H_\gamma(0,T)=J^\gamma \left(L^2(0,T) \right)$ is bijective, see e.g., Gorenflo, Luchko and Yamamoto \cite{GLY15}.
Following the treatment in \cite{GLY15,KRY20}, we know $H_\alpha(0,T)$ in \eqref{eq2.2v2} coincides with the space $J^\alpha \left(L^2(0,T) \right)$ for $\alpha\in(1,2)$,
then we extend the above pointwisely defined Caputo derivative of order $\alpha\in(1,2)$ into the operator sense.

Now we define the Caputo derivative $\partial_t^\alpha $ of order $\alpha\in (1,2)$ to be
$$
\partial_t^\alpha g = \frac{d^2}{dt^2} J^{2-\alpha} g = \frac{d}{d t}J^{2-\alpha} \frac{d}{dt}g, \quad g \in H_\alpha(0,T).
$$
The second equality can be done easily by the definition of the space $H_\alpha(0,T)$. Parallelly, we define the backward Riemann-Liouville integral operator $J^\alpha_{T-}$ by
$$
J^\alpha_{T-}g(t) =  \frac1{\Gamma(\alpha)}\int_t^T \frac{g(\tau)}{(\tau-t)^{1-\alpha}} d\tau, \quad 0<t<T,
$$
and the backward Caputo derivative by $\partial_{T-}^\alpha g(t):= -\frac{d}{d t}J^{2-\alpha}_{T-} \frac{d}{dt}g(t)$, $g\in J_{T-}^\alpha \left(L^2(0,T) \right)$.

Let $L^2(\Omega)$ be the usual $L^2$-space with the inner product $\langle \cdot,\cdot \rangle_{L^2(\Omega)}$ ($\langle\cdot,\cdot\rangle$ for short), and $H^\ell(\Omega)$, $H^k_0(\Omega)$ denote the Sobolev spaces. We denote
\begin{equation}\label{def-A0}
A_0  u(x):= -\sum_{i,j=1}^d \partial_{x_i} \left(a_{ij} \partial_{x_j}u(x) \right), \quad \forall \, u\in H_0^1(\Omega) \cap H^2(\Omega).
\end{equation}
 Then the operator $A_0$ is symmetric uniformly elliptic.
 Let $\{\lambda_n,\varphi_n\}_{n=1}^\infty$ be the Dirichlet eigensystem of $A_0$, where $0<\lambda_1\le \lambda_2\le \cdots $ and the corresponding eigenfunctions $\{\varphi_n\}_{n=1}^\infty$ forms an orthonormal basis of $L^2(\Omega)$. We can define the fractional power $A_0^\gamma$ as follows,
$$
A_0^\gamma u := \sum_{n=1}^\infty \lambda_n^\gamma \langle u,\varphi_n \rangle_{L^2(\Omega)} \varphi_n,\quad u\in D(A^\gamma_0),
$$
where $D(A_0^\gamma)$ is a Hilbert space
$$
D(A_0^\gamma) := \left\{\psi\in L^2(\Omega);  \sum_{n=1}^\infty \lambda_n^{2\gamma} \left|\langle u,\varphi_n \rangle_{L^2(\Omega)} \right|^2 < \infty \right\}
$$
with the norm
$$
\|\psi\|_{D(A_0^\gamma)} := \left[ \sum_{n=1}^\infty \lambda_n^{2\gamma}  \left| \left\langle u,\varphi_n  \right\rangle_{L^2(\Omega)} \right|^2 \right]^{\frac12}.
$$
 We have $D(A_0^\gamma) \subset H^{2\gamma}(\Omega)$ ($\gamma>0$). More specially, we have $D(A_0)=H^2(\Omega) \cap H_0^1(\Omega)$ and $D \left(A_0^{\frac12} \right)=H_0^1(\Omega)$.

\subsection{Adjoint system and bilinear form}
In this subsection, we first define the corresponding adjoint system to the problem \eqref{eq-gov}. We define the following backward problem
\begin{align}\label{eq4.1v3}
\begin{cases}
\partial_T^\alpha v - A_0(x) v + \nabla \cdot(B(x) v) - c(x)v = 0 &\text{ in } \Omega \times (0, T), \\
J_T^{1- \alpha} v(\cdot,T) = J_T^{1- \alpha} v_t(\cdot,T) = 0 &\text{ in } \Omega,\\
v(x,t) = \varphi(x,t) &\text{ on } \partial \Omega \times (0, T),
\end{cases}
\end{align}
where $\varphi \in C^\infty(\partial\Omega\times[0,T])$ s.t. supp\,$\varphi \subseteq \Gamma$, $\Gamma$ is the observation boundary which is a nonempty and open subset of $\partial \Omega$. The above problem is called the adjoint system to the original problem \eqref{eq-gov}.

We next introduce a bilinear form $\mathcal{B}_g(f, \varphi): L^2(\Omega)\times C^\infty(\partial\Omega\times[0,T]) \to \mathbb{R}$ in terms of $g$ by
\begin{equation}\label{def-B}
\mathcal{B}_g (f, \varphi) : = \int_0^T \langle f g, v[\varphi] \rangle_{L^2(\Omega)} \,\mathrm{d}t,
\end{equation}
where $v[\varphi]$ solves the problem \eqref{eq4.1v3} with $\varphi \in C^\infty(\partial\Omega\times[0,T])$ and $\text{supp} \varphi \subseteq \Gamma$.  By the above defined bilinear form, we will equip $L^2(\Omega)$ with a new suitable topology via the following functional $\| \cdot \|_{\mathcal B}: L^2(\Omega) \to \mathbb{R}^+$,
\begin{align}\label{eq2.3v3}
\| f \|_{\mathcal B}: = \sup\limits_{  \left\| \varphi  \right\|_{L^2(\Omega) }= 1}  \left|\mathcal{B}_g (f, \varphi) \right|,\quad f\in L^2(\Omega) .
\end{align}
In Section \ref{sec-stabi} we will prove the above defined functional is actually a norm of $L^2(\Omega)$.

The remaining part of this article is organized as follows. In section \ref{sec-pre},
we will give some basic fractional calculus in this paper. Section \ref{sec-fp}
is devoted to the wellposedness of the forward problem. The $t$-analyticity of
the solution to the problem \eqref{eq-gov} with $F=0$ will be discussed in Section \ref{sec-analy}.
In section \ref{sec-isp}, the uniqueness of the inverse problem will be proved by the new established unique continuation principle.
In Section \ref{sec-stabi}, on the basis of the uniqueness result, by equipping a suitable norm via the adjoint system to the problem \eqref{eq-gov},
we prove the Lipschitz continuous dependency of the solution of the inverse source problem with respect to the observation data on the partial boundary.

\section{Wellposedness for forward problem}\label{sec-fp}
In this section, we will give the definition of the mild solutions to \eqref{eq-gov} and consider the unique existence and the regularity estimate for the mild solution in the case where the coefficients, initial values and source term satisfy the following conditions:
\begin{assumption}
\label{ass-coef}
\hfill
\begin{itemize}
\item Coefficients such that $B,c\in W^{1,\infty}(\Omega)$;
\item The initial values $a\in D(A_0)$, $b\in H_0^1(\Omega)$ and the source $F\in L^2 \left(0,T;H_0^1(\Omega) \right)$.
\end{itemize}
\end{assumption}
Due to the technical difference, we first treat the non-advection case $B = 0$ and $c=0$ of \eqref{eq-gov} in Subsection \ref{ssec-symmetric}, and then proceed to the general situation in Subsection \ref{ssec-nonsym}. To this end, we will show several useful lemmata which are related to the fractional integral and derivative introduced in the above section and they will be used in the forthcoming discussion. The first one is about the Mittag-Leffler functions:
$$
E_{\alpha,\beta}(z) := \sum_{k=0}^\infty \frac{z^k}{\Gamma(\alpha k +\beta)},\quad z\in\mathbb C,
$$
where $\alpha>0$ and $\beta\in\mathbb R$. It is known that  $E_{\alpha,\beta}(z)$ is an entire function in $z\in\mathbb C$. We first state the following lemma.
\begin{lemma}\label{lem-ml}\cite[Theorem 1.4 on p.33]{P99}
Let $\alpha>0$ and $\beta\in \mathbb R$ be arbitrary. We suppose that $\mu$ is such that $\frac{\pi\alpha}{2} < \mu < \min\left\{ \pi,\pi\alpha \right\}$. Then there exists a constant $C=C(\alpha,\beta,\mu)>0$ such that
\begin{align}\label{eq3.0v5}
\left|E_{\alpha,\beta}(z) \right| \le \frac{C}{1+|z|},\quad \mu\le |\arg z| \le \pi.
\end{align}
\end{lemma}
\begin{lemma}\label{lem-ml'}
Let $\alpha>0$, $\lambda\in\mathbb R$ be arbitrarily fixed. We have
$$
\frac{d}{dt} \left(tE_{\alpha,2} \left(-\lambda t^\alpha \right) \right) = E_{\alpha,1} \left(-\lambda t^\alpha \right),
$$
$$
\frac{d}{dt} \left(E_{\alpha,1} \left(-\lambda t^\alpha \right) \right) = -\lambda t^{\alpha-1} E_{\alpha,\alpha} \left(-\lambda t^\alpha \right),
$$
and
$$
\frac{d}{dt}J^{2-\alpha}\frac{d}{dt} \left(tE_{\alpha,2} \left(-\lambda t^\alpha \right) - t \right) = -\lambda tE_{\alpha,2} \left(-\lambda t^\alpha \right),
$$
$$
\frac{d}{dt}J^{2-\alpha}\frac{d}{dt} \left(E_{\alpha,1} \left(-\lambda t^\alpha \right)\right) = -\lambda E_{\alpha,1} \left(-\lambda t^\alpha \right).
$$
\end{lemma}
This lemma can be done by a direct calculation, so we omit the proof.

\subsection{The case of $B=0$ and $c=0$}\label{ssec-symmetric}
In this subsection, we assume $B=0$ and $c=0$, and we consider the wellposedness of the following initial boundary value problems
\begin{equation}\label{eq-A0}
\begin{cases}
\partial_t^\alpha u(t) + A_0(x) u(t) = F &\text{ in } \Omega\times(0,T),\\
u(x,t) = a(x) ,\quad \partial_t u(x,t)=b(x) & \text{ in } \Omega\times\{0\},\\
u(x,t) = 0, &\text{ on } \partial \Omega \times (0, T),
\end{cases}
\end{equation}
where the operator $A_0$ is defined by \eqref{def-A0}. By Theorem 2.3 in \cite{SY11}, we can give an integral equation which is equivalent to \eqref{eq-A0}:
$$
\begin{aligned}
u(t)=&\sum_{n=1}^\infty  \langle a,\varphi_n \rangle E_{\alpha,1}(-\lambda_n t^\alpha)\varphi_n + t \sum_{n=1}^\infty \langle b,\varphi_n \rangle E_{\alpha,2}(-\lambda_n t^\alpha)\varphi_n
\\
&+\sum_{n=1}^\infty \left[ \int_0^t (t-\tau)^{\alpha-1} E_{\alpha,\alpha}(-\lambda_n(t-\tau)^\alpha) \langle F(\tau),\varphi_n \rangle_{L^2(\Omega)} d\tau \right] \varphi_n.
\end{aligned}
$$
Now we define the solution operators $S_j(z):L^2(\Omega)\rightarrow L^2(\Omega)$  for $z\in\{z\in\C\setminus\{0\};\, |\arg z|<\frac{\pi}{2}\}$, $j=1,2$ by
\begin{equation}\label{def-S(t)}
\begin{aligned}
S_1(z)a:=&\sum_{n=1}^\infty \langle a,\varphi_n \rangle_{L^2(\Omega)} E_{\alpha,1}(-\lambda_nz^{\alpha}) \varphi_n,
\quad a\in L^2(\Omega),\\
S_2(z)a:=&t\sum_{n=1}^\infty \langle b,\varphi_n \rangle_{L^2(\Omega)} tE_{\alpha,2}(-\lambda_nz^{\alpha}) \varphi_n,
\quad b\in L^2(\Omega).
\end{aligned}
\end{equation}
Moreover, from the properties of the Mittag-Leffler functions in Lemma \ref{lem-ml}, by differentiating the above $S_1(z)$ term-wisely, we get
$$
S_1'(z)a =-\sum_{n=1}^\infty \lambda_n \langle a,\varphi_n \rangle_{L^2(\Omega)} z^{\alpha-1}E_{\alpha,\alpha}(-\lambda_nz^{\alpha}) \varphi_n,
\quad a\in L^2(\Omega).
$$
We have the following estimates for the solution operators $S_1(z)$ and $S_2(z)$.
\begin{lemma}
\label{lem-S}
Let $\gamma\in [0,1]$ and suppose $z\in \mathbb C\setminus\{0 \}$ satisfying $|\arg z|<\frac\pi2$, we have the following estimates for the solution operators $S_1(z)$ and $S_2(z)$:
\begin{equation}\label{esti-S}
\begin{aligned}
\|A_0^\gamma S_1(z) \|_{\LL\to\LL} &\le C|z|^{-\al\gamma},\\
\|A_0^{\gamma}S_2(z)\|_{\LL\to\LL} &\le C|z|^{1-\alpha\gamma},\\
\|A_0^{-\gamma}S_1'(z)\|_{\LL\to\LL} &\le C|z|^{\alpha\gamma-1},
\end{aligned}
\end{equation}
where the constant $C>0$ depends only on the coefficients $a_{ij}$, $d$, $\Omega$, $\alpha$ and $\gamma$.
\end{lemma}
\begin{proof}
By the definition of the operator $S_1(t)$, and \eqref{eq3.0v5}, for any $a\in L^2(\Omega)$, we have
\begin{align*}
 \left\|A_0^\gamma S_1(z)a \right\|_\LL^2
=& \sum_{n=1}^\infty  \left|\langle a,\varphi_n \rangle \right|^2 \lambda_n^{2\gamma}   \left|E_{\alpha,1} \left(-\lambda_nz^{\alpha} \right) \right|^2
\\
\le &C \sum_{n=1}^\infty |\langle a,\varphi_n \rangle|^2 \frac{\lambda_n^{2\gamma}  |z|^{2\alpha\gamma}}{(1+\lambda_n |z|^\alpha)^2}   |z|^{-2\alpha\gamma}
\le C|z|^{-2\alpha\gamma} \|a\|_\LL^2,
\end{align*}
hence
$$
 \|A_0^\gamma S_1(z)\|_{\LL\to \LL} \le C|z|^{-\alpha\gamma},\quad |\arg z| <\frac{\pi}2.
$$
For $\gamma\in [0,1)$, we can estimate $A_0^{-1}S_1'(z)$ as follows.
\begin{align*}
 &\| A_0^{-\gamma}S_1'(z)a\|_\LL^2  \\
 \le&  \sum_{n=1}^\infty |z|^{2\alpha-2} \lambda_n^{2(1-\gamma)}|E_{\alpha,\alpha}(-\lambda_n z^\alpha)|^2 \left|\langle a,\varphi_n\rangle \right|^2
 \le C \sum_{n=1}^\infty \frac{\lambda_n^{2(1-\gamma)}|z|^{2\alpha-2}}{(1+\lambda_n |z|^\alpha)^2} \left|\langle a,\varphi_n\rangle \right|^2 \\
=& C|z|^{2\alpha\gamma-2} \sum_{n=1}^\infty \left[\frac{\lambda_n^{1-\gamma} |z|^{\alpha(1-\gamma)}}{1+\lambda_n |z|^\alpha} \right]^2 \left|\langle a,\varphi_n\rangle \right|^2
\le C|z|^{2\alpha\gamma-2} \|a\|_\LL^2,
\end{align*}
which implies
\begin{align*}
 & \left\| A_0^{-\gamma}S_1'(z) \right\|_{\LL\to\LL}
\le C|z|^{\alpha\gamma-1}.
\end{align*}
Again by \eqref{eq3.0v5}, we can similarly get the the estimate of $S_2(z)$:
\begin{align*}
 \|A_0^\gamma S_2(z)a\|_\LL^2
=& |z|^2  \sum_{n=1}^\infty |\langle b,\varphi_n \rangle|^2 \lambda_n^{2\gamma}  |E_{\alpha,2}(-\lambda_nz^{\alpha})|^2
\\
\le &C|z|^2 \sum_{n=1}^\infty | \langle b,\varphi_n \rangle|^2 \frac{\lambda_n^{2\gamma}  |z|^{2\alpha\gamma}}{(1+\lambda_n |z|^\alpha)^2}   |z|^{-2\alpha\gamma}
\le C|z|^{2-2\alpha\gamma} \|a\|_\LL^2.
\end{align*}
Thus, we obtain the second inequality in \eqref{esti-S}.
\end{proof}
By the notations of the solution operators $S_1(z)$ and $S_2(z)$, we can rephrase the integral representation formula of the solution $u$ to the problem \eqref{eq-gov} as follows.
\begin{equation}
\label{eq-int_u}
\begin{aligned}
u(t) =& S_1(t)a + S_2(t)b - \int^t_0 A_0^{-1}S_1'(t-\tau)F(\tau) d\tau,\quad 0<t<T.
\end{aligned}
\end{equation}
We will show the above representation formula \eqref{eq-int_u} of the solution $u$ is well defined. Indeed, we have the following lemma.
\begin{lemma}\label{lem-fp}
Under the assumption \ref{ass-coef}, the problem \eqref{eq-A0} admits a unique solution $u \in L^2(0,T; D(A_0))$ such that $u-a-tb\in H_\alpha(0,T;L^2(\Omega))$ and
$$
\|u\|_{L^2(0,T; D(A_0))} + \|u-a-tb\|_{H_\alpha(0,T;L^2(\Omega))} \le C\|a\|_{H^2(\Omega)} + C\|b\|_{H^1(\Omega)} + C\|F\|_{L^2(0,T;H^1(\Omega))}.
$$
\end{lemma}
Here and henceforth, $C>0$ is a constant which is independent of $t$, $T$, $a$ and $u$, but may depend on $\alpha$, $d$, $\Omega$ and the coefficients of $A$.
\begin{proof}
Taking the operator $A_0$ on both sides of \eqref{eq-int_u}, from the estimates in Lemma \ref{lem-S} for the solution operators, it follows that
\begin{align*}
\|A_0 u(t)\|
\le& \|S_1(t) A_0 a\| +  \left\|A_0^{\frac12} S_2(t) A_0^{\frac12}b \right\| +  \int^t_0  \left\|A_0^{\frac12-1}S_1'(t-\tau) A_0^{\frac12}F(\tau) \right\| d\tau
\\
\le& C\|A_0a\| + Ct^{1-\frac\alpha2} \left\|A_0^{\frac12}b \right\| + C\int_0^t (t-\tau)^{\frac\alpha2-1}  \left\|A_0^{\frac12}F(\tau) \right\| d\tau,
\end{align*}
from which we further conclude from the Young inequality for the convolution that
\begin{align*}
\int_0^T \|A_0 u(t)\|^2 dt
\le& C \|a\|_{H^2(\Omega)}^2 + C\int_0^T t^{2-\alpha}dt \|b\|_{H^1(\Omega)}^2
\\
&+ C \left( \int_0^T t^{\frac\alpha2-1} dt\right)^2 \int_0^T\|F(t)\|_{H^1(\Omega)}^2 dt.
\end{align*}
By $1<\alpha<2$, we see that the above integrands are integrable and then we get
\begin{align*}
\int_0^T \|A_0 u(t)\|^2 dt
\le& C\|a\|_{H^2(\Omega)}^2 + CT^{2-\alpha} \|b\|_{H^1(\Omega)}^2 +CT^{\alpha} \int_0^T\|F(t)\|_{H^1(\Omega)}^2 dt,
\end{align*}
hence that
\begin{align*}
\|u(t)\|_{L^2(0,T;D(A_0))}
\le C \|a\|_{H^2(\Omega)} + CT^{1-\frac\alpha2} \|b\|_{H^1(\Omega)}^2 + CT^{\frac\alpha2} \|F\|_{L^2(0,T;H^1(\Omega))}.
\end{align*}
Next, we will evaluate the function $u-a-tb$. We will first consider $S_1(t) a -a$. From the formula in Lemma \ref{lem-ml'}, we have
$$
\frac{\partial}{\partial t} J^{2-\alpha} \frac{\partial}{\partial t} (S_1(t) a-a)
= \sum_{n=1}^\infty \lambda_n \langle a,\varphi_n \rangle E_{\alpha,1}(-\lambda_n t^\alpha) \varphi_n = A_0 S_1(t) a.
$$
We immediately get $\frac{\partial}{\partial t} J^{2-\alpha} \frac{\partial}{\partial t} (S_1(t) a-a) \in L^2(0,T;L^2(\Omega))$ and
$$
\left\| \frac{\partial}{\partial t} J^{2-\alpha} \frac{\partial}{\partial t} (S_1(t) a-a) \right\|_{L^2(0,T;L^2(\Omega))} \le C\|a\|_{H^2(\Omega)},
$$
that is $S_1(t) a-a \in H_\alpha(0,T;L^2(\Omega))$ and
$$
\|S_1(t) a - a \|_{H_\alpha(0,T;L^2(\Omega))} \le C\|a\|_{H^2(\Omega)}.
$$
Similarly, again from Lemma \ref{lem-ml'}, we see that
$$
\frac{\partial}{\partial t} J^{2-\alpha} \frac{\partial}{\partial t} (S_2(t) b-tb) = A_0 S_2(t) b \in L^2(0,T;L^2(\Omega))
$$
and
$$
\frac{\partial}{\partial t} J^{2-\alpha} \frac{\partial}{\partial t} \int_0^t A_0^{-1} S_1'(t-\tau) F(\tau) d\tau
= \int_0^t A_0^{\frac12 -1} S_1'(t-\tau) A_0^{\frac12}F(\tau) d\tau.
$$
Consequently,
$$
\|S_2(t) b - b \|_{H_\alpha(0,T;L^2(\Omega))} \le T^{1-\frac\alpha2}\|b\|_{H^1(\Omega)}.
$$
and
$$
\left\| \int_0^t A_0^{-1} S_1'(t-\tau) F(\tau) d\tau \right\|_{H_\alpha(0,T;L^2(\Omega))} \le CT^{\frac\alpha2}\|F\|_{L^2(0,T;H^1(\Omega))}.
$$
Finally, we see that $u-a-tb \in H_\alpha(0,T;L^2(\Omega))$ and the following inequality holds, 
  $$
  \|u-a-tb\|_{H_\alpha(0,T;L^2(\Omega))} \le C \|a\|_{H^2(\Omega)} + CT^{1-\frac\alpha2} \|b\|_{H^1(\Omega)}^2 + CT^{\frac\alpha2} \|F\|_{L^2(0,T;H^1(\Omega))}.
  $$
\end{proof}

\subsection{The case of $B\neq0$ or $c\neq0$}\label{ssec-nonsym}
We will take the lower order term of $u$, say, $B\cdot\nabla u$ and $cu$, as new source terms and consider the wellposedness of the following initial boundary value problems
\begin{equation}\label{eq-gov'}
\begin{cases}
\partial_t^\alpha u(t) + A_0(x) u(t) = -B(x)\cdot\nabla u(t) - c(x) u(t) +F &\text{ in } \Omega\times(0,T),\\
u(x,t) = a(x) ,\quad \partial_t u(x,t)=b(x) & \text{ in } \Omega\times\{0\},\\
u(x,t) = 0, &\text{ on } \partial \Omega \times (0, T),
\end{cases}
\end{equation}
where the operator $A_0$ is defined by \eqref{def-A0}. By taking the nonsymmetric term and zeroth order term as the new source terms, we can argue similarly as that in the proof in the above subsection to see that the solution formally satisfies the integral equation
\begin{equation}\label{eq-int_u'}
\begin{aligned}
u(t) =& S_1(t)a + S_2(t)b - \int^t_0 A_0^{-1}S_1'(t-\tau)F(\tau) d\tau\\
&+\int^t_0 A_0^{-1}S_1'(t-\tau)(B\cdot\nabla u(\tau)+cu(\tau)) d\tau,\quad 0<t<T.
\end{aligned}
\end{equation}
Based on the integral equation \eqref{eq-int_u'}, we can show
the wellposedness of the solution to \eqref{eq-gov}.
\begin{lemma}
\label{lem-fp'}
Under the assumption \ref{ass-coef}.Then the problem \eqref{eq-gov} admits a unique mild solution $u \in L^2(0,T; D(A_0)) \cap C([0,T];L^2(\Omega))$ such that $u-a-tb \in H_\alpha(0,T; L^2(\Omega))$ and
\begin{equation}\label{esti-fp'}
\|u\|_{L^2(0,T; D(A_0))} + \|u-a-tb\|_{H_\alpha(0,T;L^2(\Omega))} \le C(\|a\|_{H^2(\Omega)}+ \|b\|_{H^1(\Omega)} + \|F\|_{L^2(0,T;H^1(\Omega))}).
\end{equation}
\end{lemma}
\begin{proof}
We rephrase the integral equation \eqref{eq-int_u'} of $u$ into the operator form
$$
u(\,\cdot\,,t)=\Psi(\,\cdot\,,t) + \mathcal N u(\cdot,t),
$$
where
\begin{align*}
\Psi(\,\cdot\,,t)&:= S_1(t)a + S_2(t)b - \int^t_0 A_0^{-1}S_1'(t-\tau)F(\tau) d\tau, \\
\mathcal N u(\,\cdot\,,t) &:= \int_0^tS_1(t-\tau) (B\cdot\nb u(\,\cdot\,,\tau) + cu(\tau)) d\tau,\quad t\in (0,T).
\end{align*}
In the sequel, we will prove $\Psi\in H^1(0,T;L^2(\Omega)) \cap L^2(0,T;D(A_0))$ and the operator $\mathcal N:L^2(0,T;H_0^1(\Omega))\to L^2(0,T;H_0^1(\Omega))$ is compact. To this end, recalling the argument used in Section \ref{ssec-symmetric}, we have $\Psi\in L^2(0,T;D(A_0))$ such that
$$
\|\Psi\|_{L^2(0,T;D(A_0))} \le C\|a\|_{D(A_0)} + C\|b\|_{H^1(\Omega)}+ C\|F\|_{L^2(0,T;H^1(\Omega))}.
$$
To verify the assertion for $\Psi$, we only need to show $\frac{\partial}{\partial t}\Psi\in L^2(0,T;L^2(\Omega))$. Indeed, by the formulas in Lemma \ref{lem-ml'}, a direct calculation yields
\begin{align*}
\frac{\partial}{\partial t} \Psi(t)
=& -\sum_{n=1}^\infty \lambda_n \langle a,\varphi_n \rangle t^{\alpha-1} E_{\alpha,\alpha}(-\lambda_n t^\alpha) \varphi_n
+ \sum_{n=1}^\infty \langle b,\varphi_n \rangle E_{\alpha,1}(-\lambda_n t^\alpha) \varphi_n
\\
&+\sum_{n=1}^\infty \int_0^t (t-\tau)^{\alpha-2} E_{\alpha,\alpha-1}(-\lambda_n (t-\tau)^\alpha) \langle F(\tau),\varphi_n\rangle d\tau \varphi_n.
\end{align*}
By \eqref{eq3.0v5} and the Young inequality for convolution, we have
\begin{align*}
\left\|\frac{\partial}{\partial t}\Psi \right\|_{L^2(0,T;L^2(\Omega))}^2
\le& CT^{2\alpha-2} \int_0^T \sum_{n=1}^\infty \lambda_n^2 \langle a,\varphi_n \rangle^2 dt + \int_0^T \sum_{n=1}^\infty \langle b,\varphi_n \rangle^2 dt \\
 &+\sum_{n=1}^\infty \left( \int_0^T t^{\alpha-2} dt\right)^2 \int_0^T |\langle F(t),\vp_n\rangle|^2 dt,
\end{align*}
this indicates
$$
\left\|\frac{\partial}{\partial t}\Psi \right\|_{L^2(0,T;L^2(\Omega))}
\le CT^{\alpha-1} \|a\|_{D(A_0)} + C\|b\|_{L^2(\Omega)} + CT^{\alpha-1}\|F\|_{L^2(0,T;L^2(\Omega))}.
$$
Now we turn to show $\mathcal N:L^2(0,T;H_0^1(\Omega))\to L^2(0,T;H_0^1(\Omega))$ is compact operator.
In fact, by the argument used in evaluating $\frac{\partial}{\partial t} \Psi(t)$, we can show that for $u\in L^2(0,T;H_0^1(\Omega))$,
\begin{align*}
\left\|\frac{\partial}{\partial t}\mathcal N u(t) \right\|_{L^2(0,T;L^2(\Omega))}
\le CT^{\alpha-1}\|u\|_{L^2(0,T;L^2(\Omega))}.
\end{align*}
Moreover, when $\gamma\in [\frac12,1)$, by the estimates in Lemma \ref{lem-S}, we have
\begin{equation}\label{esti-Nu}
\begin{aligned}
 \left\|A_0^{\gamma} \mathcal N u(t) \right\|_{L^2(\Omega)}
\le& \int_0^t  \left\|A_0^{\gamma-1} S_1'(t-\tau) \right\|_{L^2(\Omega) \to L^2(\Omega)}  \left\|B\cdot\nabla u(\tau) + cu(\tau) \right\|_{L^2(\Omega)} d\tau
\\
\le& \int_0^t (t-\tau)^{\alpha(1-\gamma)-1} \|u(\tau)\|_{H^1(\Omega)} d\tau.
\end{aligned}
\end{equation}
By the Young inequality for convolution, we have
\begin{align*}
\|\mathcal N u(t)\|_{L^2(0,T;D(A_0^{\gamma}))}
\le  \int_0^T t^{\alpha(1-\gamma)-1} dt \left(\int_0^T \|u(t)\|_{H^1(\Omega)}^2 dt \right)^{\frac12} 
\le& C\frac{T^{\alpha(1-\gamma)}}{\alpha(1-\gamma)} \|u\|_{L^2(0,T;H^1(\Omega))}.
\end{align*}
Collecting all the above estimates for $\mathcal Nu$, we conclude that $\mathcal N$ maps $L^2(0,T;H^1(\Omega))$ into $H^1(0,T;L^2(\Omega)) \cap L^2(0,T;D(A_0^\gamma))$ with $\gamma\in(\frac12,1)$.
Moreover, since the embedded mapping from $H^1(0,T;L^2(\Omega)) \cap L^2(0,T;D(A_0^\gamma))$ to $L^2(0,T;H^1(\Omega))$ is compact, we immediately obtain the compactness of the operator $\mathcal N:L^2(0,T;H_0^1(\Omega)) \to L^2(0,T;H_0^1(\Omega))$.

Now we turn to verify $1$ is not an eigenvalue of $\mathcal N$, that is, $\mathcal N u=u$ in $L^2(0,T;H_0^1(\Omega))$ implies $u=0$. By letting $\gamma=\frac12$ in the inequality \eqref{esti-Nu}, we see that
$$
\left\|A_0^{\frac12} u(t) \right\|_{L^2(\Omega)} =  \left\|A_0^{\frac12} \mathcal N u(t) \right\|_{L^2(\Omega)} \le C\int_0^t (t-\tau)^{\frac\alpha2-1} \|u(\tau)\|_{H^1(\Omega)} d\tau,
$$
from which we further conclude from the general Gronwall inequality that $u$ must vanish in $\Omega\times(0,T)$. This finishes the proof of the above statement.

Consequently, by the Fredholm alternative, we see that $I-\mathcal N$ is bounded from $L^2(0,T;H^1(\Omega))$ into itself. Therefore
$$
u(t) = (I- \mathcal N)^{-1} \Psi(t) \in L^2 \left(0,T;H_0^1(\Omega) \right).
$$
Now we note that the right hand side of the differential equation in \eqref{eq-gov'} belongs to the space $L^2 \left(0,T;L^2(\Omega) \right)$, that is, $-B\cdot\nabla u -cu+F\in L^2 \left(0,T;L^2(\Omega) \right)$, and we can repeat the above argument to derive that the solution $u$ to the integral equation \eqref{eq-int_u'} is such that $u\in L^2(0,T;D(A_0)^\gamma)$, $\gamma\in (\frac12,1)$. Finally, by noting the estimate
$$
\begin{aligned}
\|A_0 \mathcal N u(t)\|_{L^2(\Omega)}
\le& \int_0^t  \left\|A_0^{\frac32 -\gamma-1} S_1'(t-\tau) \right\|_{L^2(\Omega) \to L^2(\Omega)}  \left\|A_0^{\gamma-\frac12} \left(B\cdot\nabla u(\tau) + cu(\tau) \right) \right\|_{L^2(\Omega)} d\tau\\
\le& \int_0^t (t-\tau)^{\alpha(\gamma-\frac12)-1} \|u(\tau)\|_{D(A_0^\gamma)} d\tau.
\end{aligned}
$$
From the Young inequality for the convolution, it follows that $\mathcal Nu\in L^2(0,T;D(A_0))$. Finally, recalling the result in the above subsection, we can obtain that the solution to the integral equation \eqref{eq-int_u'} is such that $u\in L^2(0,T;D(A_0))$ and $u-a-tb \in H_\alpha(0,T;L^2(\Omega))$. It is not difficult to check the inequality \eqref{esti-fp'} is valid by the the Gronwall inequality. We complete the proof of Lemma \ref{lem-fp'}.
\end{proof}

\section{Analyticity with respect to time}\label{sec-analy}
For the homogeneous equation, that is, $F=0$,  we have the analyticity of the solution and the growth estimate.
\begin{lemma}\label{lem-analy}
Under the assumption \ref{ass-coef}, we further let $F=0$.
The solution $u:(0,T]\rightarrow D(A_0)$ can be analytically extended to the sector $\{ z\in \mathbb C\setminus \{0\};|\arg z|<\frac{\pi}{2} \}$,
and the extention $u:(0,\infty) \rightarrow D(A_0)$ satisfies
$$
\|u(t)\|_{D(A_0)} \le C\min\{t^{-\alpha}, 1\}e^{Ct} \|a\|_{D(A_0)} + \|b\|_{H^1(\Omega)},\quad t>0.
$$
\end{lemma}
Here and henceforth, $C>0$ denotes constants which are independent of $t$, $T$, $a$ and $u$, but may depend on $\alpha$, $d$, $\Omega$ and the coefficients of $A$.
\begin{proof}
Since the solution $u$ to the IBVP \eqref{eq-gov} satisfies the integral equation \eqref{eq-int_u'}, after the change of variables, we find
\begin{align}\label{equ-int-u-standard}
u(t)=S_1(t)a + S_2(t)b + t\int_0^1 A_0^{-1}S_1'(rt) \left( B\cdot\nabla u((1-r)t)+cu((1-r)t) \right) \d r.
\end{align}
Moreover, we extend the variable $t$ in \eqref{equ-int-u-standard} from $(0,T)$ to the sector $\{z\in \mathbb C\setminus\{0\}; |\arg z|<\frac{\pi}{2}\}$, and setting $u_0=0$, we define $u_{n+1}(z)$ $(n=0,1,\dots)$ as follows:
\begin{align}
\label{def-u_n}
u_{n+1}(z) =&S_1(z)a + S_2(z)b
 -z\int_0^1 A_0^{-1}S_1'(rz) \left( B\cdot\nabla u_n+cu_n\right) ((1-r)z) \d r.
\end{align}
Following the argument in \cite{LHY20}, we conclude from the definition \eqref{def-S(t)} of the solution operators $S_1(z),S_2(z)$ and the properties of Mittag-Leffler function that $u_n(z)$ defined in \eqref{def-u_n} uniformly converges to the solution to the initial-boundary value problem \eqref{eq-gov} as $n\to\infty$ for any compact subset of the section $ \left\{z\in\mathbb C\setminus\{0\};\, |\arg z|<\frac{\pi}{2} \right\}$. The details are listed as follows.

For any nonnegative integer $n=0,1,\cdots$, taking the operator $A_0$ on both sides of \eqref{def-u_n}, we proceed by induction on $n$ to prove the inequality
\begin{align}
\label{esti-u_n}
\|u_{n+1}(z)-u_n(z)\|_{D(A_0)}
\leq M_0\frac{M^n |z|^{\frac\alpha2 n}}{\Gamma \left(\frac\alpha2 n+1 \right)},
\quad n=0,1,\cdots,
\end{align}
where the constant $M$ is independent of $T$, $t>0$, $z\in S_{\te,T}:= \{z\in S_\te; |z|\le T\}$, but may depend on $d$, $\Om$, $\al$.
Firstly, for $n=0$, using the estimate \eqref{esti-S} in Lemma \ref{lem-S} for $z\in S_{\te,T}$, it follows that
\begin{align*}
\|u_1(z) - u_0(z)\|_{D(A_0)}
\le& \|A_0 S_1(z)a\|_\LL + \|A_0 S_2(z)b\|_\LL  \\
\le& C\|a\|_{H^2(\Omega)} + CT \|b\|_{H^1(\Omega)}=:M_0.
\end{align*}
Next, for any $n\in\N$, in view of the inequality  $\| B\cdot\na v \|_\LL \le C\| v\|_{D(A_0)}$ holds for $v \in D(A_0)$ since $B\in W^{1,\infty}(\Omega)$, we derive
\begin{align*}
 & \left\| A_0^{-1}S'(r z)\big( B\cdot\na (u_n-u_{n-1}) + c(u_n-u_{n-1})\big)\big((1-r)z\big) \right\|_{D(A_0)}\\
\le& C \left\| A_0^{\frac12-1}S'(r z) \right\|_{\LL\to \LL}  \|(u_n-u_{n-1})((1-r)z)\|_{D(A_0)}\\
\le&C|rz|^{\frac\alpha2-1}  \|(u_n-u_{n-1})((1-r)z)\|_{D(A_0)} .
\end{align*}
Combining the above inequalities and the equation \eqref{def-u_n}, we can obtain
\begin{align*}
 \|u_{n+1}(z) - u_n(z)\|_{D(A_0)} 
\le& C|z|^{\frac\alpha2} \int_0^1 (1-r)^{\frac\alpha2-1} \|u_n(rz) - u_{n-1}(rz)\|_{D(A_0)}\d r.
\end{align*}
Consequently, we prove by inductive argument that
\begin{align*}
\|u_{n+1}(z) - u_n(z)\|_{D(A_0)}
\le& CM_0 |z|^{\frac\alpha2} \int_0^1 (1-r)^{\frac\alpha2-1} \frac{M^{n-1} |rz|^{\frac\alpha2 (n-1)}}{\Gamma \left(\frac\alpha2 (n-1) +1 \right)} dr
\\
=&CM_0 \Gamma \left(\frac\alpha2 \right) \frac{|z|^{\frac\alpha2 n}M^{n-1}}{\Gamma \left(\frac\alpha2 n+1 \right)},\quad z\in S_{\te,T},
\end{align*}
where in the last equality we used the relation between Gamma function and Beta function
$B(\alpha,\beta)=\frac{\Gamma(\alpha) \Gamma(\beta)}{\Gamma(\alpha + \beta)}$, $\alpha,\beta>-1$.
Finally, by choosing $M:=C\Gamma(\frac\alpha2)$, we obtain \eqref{esti-u_n}.

For any compact subset $K$ of $S_{\te,T}$, from the asymptotic behavior of the Gamma function, we see the infinite series
$\sum\limits_{n=0}^\infty \frac{ M^n |z|^{\frac\alpha2 n}}{ \Gamma(\frac\alpha2 n + 1) }$
is uniformly convergent in $K$, which implies $u(z) := \sum\limits_{n=0}^\infty \left(  u_{n+1}(z) - u_n(z) \right) $ is analytic in the compact set $K$ since each term $u_n$ is analytic in $S_\te$.  By restricting the variable $z$ to the interval $(0,T)$, we assert that $u(t)$ is the unique solution to the integral equation \eqref{eq-int_u'} with $F=0$.

As a byproduct, the above argument also indicates that the solution $u(t)$ can be analytically extended to $u(t): (0,\infty) \to D(A_0)$. Moreover, the extended $u$ admits the following growth estimate
\begin{align*}
\| A_0 u(t)\|_\LL
\le  \sum_{n=0}^\infty \| A_0 u_{n+1}(t) - A_0 u_n(t)\|_\LL 
\le& M_0\sum_{n=0}^\infty  \frac{M^n t^{\frac\alpha2 n}} {\Gamma(\frac\alpha2 n +1) }
=M_0 E_{\frac\alpha2, 1}(M t^{\frac\alpha2}),
\end{align*}
which implies for any $T>0$,
\begin{equation}
\label{esti-u*}
\|u(t)\|_{D(A_0)} \leq CM_0e^{M^{\frac2\alpha}t}
\le Ce^{M^{\frac2\alpha}T}(\|a\|_{H^2(\Omega)} + T\|b\|_{H^1(\Omega)}),\quad 0<t<T,
\end{equation}
in view of the asymptotic expansion properties of the Mittag-Leffler functions in \cite[Theorem 1.4 on p.33]{P99}. This completes the proof of the lemma.
\end{proof}

\section{Uniqueness of the inverse source problem}\label{sec-isp}
In this section, we will consider the proof of uniqueness of the determination of the source term. For this, we will transform the inverse problem of searching the unknown source term into solving the determination of the initial value of the corresponding parabolic equation through the Laplace transform, and prove the uniqueness.

\subsection{Unique continuation}
In this part, we consider the unique continuation principle of the solution of the homogeneous time-fractional diffusion diffusion-wave equation as follows
\begin{align}\label{eq-ucp}
\begin{cases}
\partial_t^\alpha u +A(x)u = 0 &\text{ in } \Omega \times (0, T), \\
u(x,0) = 0,\ u_t(x,0) = b &\text{ in } \Omega,\\
u(x,t)=0 &\text{ on } \partial\Omega\times(0, T).
\end{cases}
\end{align}
The unique continuation is a characteristic property of the solution of the fractional diffusion-wave equation, which is useful in  dealing with the inverse source problem and the approximate controllability, see Fujishiro and Yamamoto \cite{FY14}, Jiang, Li, Liu and Yamamoto \cite{JLLY17}.
\begin{lemma}[Unique continuation]\label{lem-ucp}
Let $\Gamma\subset\partial\Omega$ be an arbitrarily chosen subboundary and let $u$ be the solution to the problem \eqref{eq-ucp} with $b\in D \left(A_0^{\frac12} \right)$. Then $\partial_{\nu_A} u=0$ on $\Gamma\times(0, T)$ implies $u = 0$ in $\Omega\times(0, T)$.
\end{lemma}
\begin{proof}
It suffices to show that the initial value vanishes by the uniqueness of the solution to \eqref{eq-gov-infty}. For this, according to our assumptions and Lemma \ref{lem-analy}, we can analytically extend the solution $u:(0,T] \to D(A_0)$ to $u(t): (0,\infty) \to D(A_0)$. We still denote the extension by $u$ if no conflict occurs. Therefore, the original problem \eqref{eq-ucp} becomes
\begin{align}\label{eq-gov-infty}
\begin{cases}
\partial_t^\alpha u +A(x)u = 0 &\text{ in } \Omega \times (0, \infty), \\
u(x,0) = 0,\ u_t(x,0) = b &\text{ in } \Omega,\\
u(x,t)=0 &\text{ on } \partial\Omega\times(0, \infty)
\end{cases}
\end{align}
with the additional condition
\begin{equation}\label{eq-van}
\partial_{\nu_A}u=0\quad\mbox{in}\quad \Gamma\times(0,\infty).
\end{equation}
By the Laplace transforms (denoted by $\widehat\cdot$\,), the above differential equation can be rephrased into the transformed algebraic equation
$$
\begin{cases}
(A+s^\alpha)\widehat u(x;s)=s^{\al-2}b(x), & x\in\Omega,\\
\widehat u(x;s)=0, & x\in\partial\Omega
\end{cases}
$$
with a parameter $s>s_1$, where $s_1>0$ a sufficiently large constant. Now it is not difficult to show that
the function $\widehat u_1(x;s):=s^{2-\al}\widehat u(x;s)$ satisfies 
 the following boundary value problem for an elliptic equation
\begin{equation}\label{equ-lap-v}
\begin{cases}
(A+s^\al)\widehat u_1(x;s)=b(x), & x\in\Omega,\\
\wh u_1(x;s)=0, & x\in\partial\Omega,
\end{cases}\quad s>s_1.
\end{equation}
On the other hand, we consider an initial-boundary value problem for a parabolic equation
$$
\begin{cases}
\pa_tu_2+ A u_2=0 & \mbox{in }\Om\times(0,\infty),\\
u_2=a & \mbox{in }\Om\times\{0\},\\
u_2=0 & \mbox{on }\pa\Omega\times(0,\infty).
\end{cases}
$$
Again, applying the Laplace transform, we can similarly obtain
$$\begin{cases}
(A+\eta)\widehat u_2(x;\eta)=a(x), & x\in\Omega,\\
\widehat u_2(x;\eta)=0, & x\in\partial\Omega,
\end{cases}
$$
where the parameter $\eta>s_2$ and $s_2>0$ is a sufficiently large constant. By the change of the variable $\eta=s^\alpha$, we have
$$
\begin{cases}
(A+s^\alpha)\widehat u_2(x;s^\al)=a(x), & x\in\Omega,\\
\wh u_2(x;s^\alpha)=0, & x\in\partial\Omega,
\end{cases}\quad s^\alpha>s_2.
$$
By the uniqueness result of the boundary value problems of elliptic equations, we have $\widehat u_2$ equals to the solution $\widehat u_1$ to the elliptic problem \eqref{equ-lap-v}, that is,
$$
\widehat u_2(x;s^\al)=\widehat u_1(x;s)=s^{2-\alpha}\widehat u(x;s),\quad(x;s)\in\Omega\times\{s>s_0\},\ s_0:=\max\{s_2^{1/\alpha},s_1\}.
$$
{This combine} with the additional condition $\partial_{\nu_A}\wh u(x;s)=0$ on $\Gamma\times(s_0,\infty)$ implies that
$$
\partial_{\nu_A}\widehat u_2(x;\eta)=0,\quad(x;\eta)\in\Gamma\times\{\eta>s_0^\alpha\}.
$$
Now the uniqueness of the inverse Laplace transform indicates that $\partial_{\nu_A}u_2=0$ in $\Gamma\times(0,\infty)$. We must have $u_2=0$ in $\Omega\times(0,\infty)$ in view of the unique continuation property for parabolic equations, hence that $a=u_2(\,\cdot\,,0)=0$ in $\Omega$, which completes the proof.
\end{proof}

\subsection{Proof of Theorem \ref{thm-unique}}
Before giving the proof of our second result, we will show several useful lemmata. The first one is about the Riemann-Liouville integral.
\begin{lemma}[Convolution formula for the Riemann-Liouville fractional integral, \cite{YL21}]\label{lem-convo-RL}
Let $\alpha>0$ and $g,h\in L^2(0,T)$, then
$$
J^\alpha\left(\int_0^t g(\tau) h(t-\tau) d\tau \right)
=\int_0^t  g(\tau) (J^\alpha h)(t-\tau) d\tau,\quad 0<t<T.
$$
\end{lemma}

\begin{lemma}[Duhamel principle]\label{lem-duha}
Letting $a=b=0$ and $F=g(t)f(x)$ in the problem \eqref{eq-gov} such that $g\in L^2[0,T]$ and $f\in H_0^1(\Omega)$.
Then the solution $u$ to the problem satisfies the following representation
\begin{equation}\label{eq-v}
u(t) = \int_0^t \rho(t-\tau) v(\tau) d\tau,\quad t\in (0,T),
\end{equation}
where $J^{2-\alpha} \rho(t) = g(t)$ and $v$ is the solution to the following problem
\begin{align}\label{eq-duha}
\begin{cases}
\partial_t^\alpha v +A(x)v = 0 &\text{ in } \Omega \times (0, T), \\
v(x,0) = 0,\ v_t(x,0) = f &\text{ in } \Omega,\\
v(x,t)=0 &\text{ on } \partial\Omega\times(0, T).
\end{cases}
\end{align}
\end{lemma}
\begin{proof}
We denote the right hand side of \eqref{eq-v} as $\widetilde u$, and we only need to show that $\widetilde u$ satisfies the problem \eqref{eq-gov} in view of the uniqueness result in Lemma \ref{lem-fp'} in Section \ref{sec-fp}. For this, we make the following calculation
\begin{align*}
\partial_t^\alpha \tilde u(t)
=& \frac{\partial}{\partial t} J^{2-\alpha} \frac{\partial}{\partial t} \left(\int_0^t \rho(t-\tau) \tau d\tau f + \int_0^t \rho(\tau) (v(t-\tau) - (t-\tau) f) d\tau \right)
\\
=&\frac{\partial}{\partial t} J^{2-\alpha} \left(\int_0^t \rho(t-\tau) d\tau f \right)+
\frac{\partial}{\partial t} J^{2-\alpha} \left(\int_0^t \rho(\tau) \left(\frac{\partial}{\partial t}(v(t-\tau) - (t-\tau)f) \right) d\tau \right).
\end{align*}
Here in the last equality we used the fact that $v(\cdot,0)=0$ in $\Omega$. Moreover, by Lemma \ref{lem-convo-RL}, and noting $J^{2-\alpha} \rho = g$,
we see that $\partial_t^\alpha \tilde u(t)$ can be rephrased as
\begin{align*}
\partial_t^\alpha \tilde u(t)=& g(t) f(x) +  \int_0^t \rho(\tau) \frac{\partial}{\partial t} J^{2-\alpha}\frac{\partial}{\partial t}\left(v(t-\tau) - (t-\tau)f \right) d\tau\\
=&g(t) f(x) + \int_0^t \rho(t-\tau)\partial_t^\alpha v(\tau) d\tau,\quad 0<t<T.
\end{align*}
Finally, by \eqref{eq-duha}, we see that
\begin{align*}
\partial_t^\alpha \tilde u(t)=g(t) f(x) - A(x) \tilde u(t),\quad 0<t<T.
\end{align*}
This completes the proof of the lemma.
\end{proof}
On the basis of the unique continuation of the solution established in the above subsection, we are ready to give the proof of the uniqueness of our inverse source problem.
\begin{proof}[Proof of Theorem \ref{thm-unique}]
By $\partial_{\nu_A} u(x,t)= 0$ on $\Gamma\times(0,T) $, and Lemma \ref{lem-duha}, we have
\begin{align*}
\int_0^t \rho( \tau) \partial_{\nu_A} v(t - \tau) \,d\tau = 0, \quad x \in \Gamma,\ t \in (0, T),
\end{align*}
where $v$ is the solution to the problem \eqref{eq-duha}. Moreover, multiplying the Riemann-Liouville integral $J^{2-\alpha}$ on both sides of the above equation,
and then taking $t$-derivative, we have
$$
g(0) \partial_{\nu_A} v(t) + \int_0^t g'(t - \tau) \partial_{\nu_A} v(\tau) \,d\tau = 0,\quad x\in\Gamma,\ t\in(0,T).
$$
Now since $g(0) \ne 0$, we conclude from Gronwall's inequality that
\begin{align*}
\partial_{\nu_A} v(x,t) = 0, \quad x \in \Gamma,\ t \in (0,T),
\end{align*}
from which we see the solution $v$ to the problem \eqref{eq-duha} must vanish in $\Omega\times(0,T)$ by the unique continuation for the fractional diffusion-wave equation in Lemma \ref{lem-ucp}, hence $f \equiv 0$. We finish the proof of the theorem.
\end{proof}

\section{Stability of the inverse source problem} \label{sec-stabi}
In this section, following the arguments in \cite{LY06, LL20}, we establish a Lipschitz stability for determining the source.

\subsection{An integral identity}
For proving the second main result, we will give an integral identity with aid of an adjoint problem, which reflects a corresponding relation of varied of unknown source functions with changes of the boundary values and additional observations.
To this end, we will first show a useful relation between Riemman-Liouville integral and the backward Riemman-Liouville integral.

\begin{lemma}\label{lem-J_T}
Suppose $f,g \in L^2 (0, T)$, then
\begin{align*}
\int_0^T (J^\alpha f)(t) g(t) \,\mathrm{d}t = \int_0^T f(t) J_T^\alpha g(t) \,\mathrm{d}t.
\end{align*}
\end{lemma}
The proof can be done by direct calculation in view of the Fubini lemma or one can refer to \cite{FY14}, so we omit the proof of this lemma.
Based on the above lemma, we can derive the following integral identity.
\begin{lemma}\label{le4.2v3}
Consider $u$ the solution of \eqref{eq-gov} with the source term $F(x,t)=f(x)g(t)$ and initial values $a=b=0$. Then we have
\begin{align}\label{eq4.2v3}
\int_0^T \int_\Gamma ( \partial_{\nu_A} u) \varphi \,\mathrm{d}x \mathrm{d}t = \int_0^T  \left\langle f(\cdot) g(t), v[\varphi] \right\rangle_{L^2(\Omega)} \,\mathrm{d}t,
\end{align}
where $\varphi\in C^\infty(\partial\Omega\times(0,T))$ with supp $\varphi\subseteq \Gamma$ and $v[\varphi]$ is the solution to the problem \eqref{eq4.1v3}.
\end{lemma}
\begin{proof}
This can be done by using Lemma \ref{lem-J_T}. In fact, we multiply $v[\varphi]$ on both sides of the equation in \eqref{eq-gov}, and integrate on $\Omega \times (0,T)$, and we conclude from Lemma \ref{lem-J_T} and integration by parts that
\begin{align*}
 \int_0^T \langle fg , v[\varphi] \rangle \,\mathrm{d}t
 =& \ \int_0^T \langle \partial_t^\alpha u, v[\varphi]\rangle \,\mathrm{d}t + \int_0^T \langle (A_0+c) u, v[\varphi] \rangle \,\mathrm{d}t + \int_0^T \langle B\cdot\nabla u, v[\varphi] \rangle \,\mathrm{d}t \\
=& - \int_0^T \langle u, \partial_T^\alpha v[\varphi] - (A_0+c) v[\varphi] + \nabla \cdot( B v[\varphi]) \rangle \,\mathrm{d}t  + \int_0^T \int_\Gamma \partial_{\nu_A} u \cdot \varphi \,\mathrm{d}x \mathrm{d}t.
\end{align*}
Together with \eqref{eq4.1v3}, we can further get
\begin{align*}
\int_0^T \langle fg , v[\varphi] \rangle \,\mathrm{d}t
= \int_0^T \int_\Gamma ( \partial_{\nu_A} u) \varphi \,\mathrm{d}x \mathrm{d}t.
\end{align*}
\end{proof}

\subsection{Construction of stability} \label{se4.2v3}
On the basis of the integral identity and noting the uniqueness of the inverse source problem established in the previous section, we can show that the functional $\| \cdot \|_{\mathcal B}: L^2(\Omega) \to \mathbb{R}^+$ defined \eqref{eq2.3v3} is a new norm of $L^2(\Omega)$.

\begin{lemma}\label{le4.3v3}
$\| \cdot \|_{\mathcal B}$ defined by \eqref{eq2.3v3} is a norm on $L^2(\Omega)$.
\end{lemma}

\begin{proof}
Step 1. $ \| c  f  \|_{\mathcal B} = |c| \|  f  \|_{\mathcal B}$, $\forall\, c \in \mathbb{R}$,

From the definition of the function $\|\cdot\|_{\mathcal B}$, we have
\begin{align*}
\| c  f  \|_{\mathcal B} = \sup\limits_{ \| \varphi \|_{L^2(\Omega)} = 1}  | \mathcal{B}_{ g } ( c  f , \varphi) |.
\end{align*}
From the bi-linearity of $\mathcal{B}_{ g }$, it follows that
\begin{align*}
\mathcal{B}_{ g } (c  f , \varphi) = c \mathcal{B}_{ g } (  f , \varphi),
\end{align*}
which immediately implies
\begin{align*}
\| c  f  \|_{\mathcal B}= \sup\limits_{ \| \varphi \|_{L^2(\Omega)} = 1} c  \left|\mathcal{B}_{ g } (  f , \varphi) \right|
= c \| f  \|_{\mathcal B}.
\end{align*}

Step 2.
$\|  f  +  h  \|_{\mathcal B} \le \| f  \| _{\mathcal B}+ \| h  \|_{\mathcal B}$.

By the triangle inequality, a direct calculation yields
\begin{align*}
\left| \mathcal{B}_{ g } (  f  +  h  , \varphi) \right|
\le  \left|\mathcal{B}_{ g } (  f  , \varphi) \right| + \left| \mathcal{B}_{ g } (  h  , \varphi) \right|,
\end{align*}
then
\begin{align*}
\sup\limits_{ \| \varphi \|_{L^2(\Omega) } = 1} |\mathcal{B}_{ g } (  h  +  f , \varphi) |
\lesssim \sup\limits_{ \| \varphi \|_{L^2(\Omega)} = 1}   |\mathcal{B}_{ g } (  f , \varphi)|
+ \sup\limits_{\| \varphi \|_{L^2(\Omega)= 1}}  |\mathcal{B}_{ g } (  h , \varphi)|,
\end{align*}
that is $\|  f  +  h  \|_{\mathcal B} \le \|  f  \|_{\mathcal B} + \|  h  \|_{\mathcal B}$.

Step 3. $\|  f  \|_{\mathcal B}  = 0 \Rightarrow   f  = 0 \text{ on } \Omega.$

We note that $\|  f  \|_{\mathcal B} = 0$ implies
\begin{align*}
| \mathcal{B}_{ g } (  f , \varphi) | = 0,\quad \forall\, \varphi\in C^\infty(\partial\Omega\times(0,T)) \text{ with supp} \varphi \subset \Gamma.
\end{align*}
Then from the definition of $\mathcal{B}_{ g } $, we have
\begin{align*}
\int_0^T \langle fg, v[ \varphi] \rangle_{L^2(\Omega)}  \,\mathrm{d}t = 0, \quad\forall\, \varphi\in C^\infty(\partial\Omega\times(0,T))  \text{ with supp} \varphi \subset \Gamma.
\end{align*}
By \eqref{eq4.2v3}, it follows that
\begin{align*}
\int_0^T \int_{\Gamma} ( \partial_{\nu_A} u ) \varphi \,\mathrm{d}x \mathrm{d}t = 0,\quad\forall\, \varphi\in C^\infty(\partial\Omega\times(0,T))  \text{ with supp} \varphi \subset \Gamma.
\end{align*}
Thus $\partial_\nu u = 0$ on $\Gamma \times (0, T)$. Now from the uniqueness result for the inverse source problem established above,
 it follows that $\partial_{\nu_A} u =0$ on $\Gamma \times (0,T)$, and therefore $ f  = 0$.

 Collecting the above results, we see $\| \cdot \| $ is a norm on $L^2(\Omega) $.
\end{proof}
We are now in a position to give the proof of our second main result.
\begin{proof}[Proof of Theorem \ref{thm-stabi}]
From the integral identity \eqref{eq4.2v3}, it follows that
\begin{align*}
\mathcal{B}_{ g } (  f_1  -  f_2 , \varphi)
= \int_0^T \langle f_1g-  f_2 g, v[ \varphi] \rangle_{L^2(\Omega)} \,\mathrm{d}t 
=& \int_0^T \int_\Gamma ( \partial_{\nu_A} u_1 - \partial_{\nu_A} u_2 ) \varphi \,\mathrm{d}x \mathrm{d}t.
\end{align*}
By the Holder inequality, we have
\begin{align*}
|\mathcal{B}_{ g } (  f_1  -  f_2  , \varphi) |
 \le \left| \int_0^T \int_\Gamma ( \partial_{\nu_A} u_1 - \partial_{\nu_A} u_2 ) \varphi \,\mathrm{d}x \mathrm{d}t \right|
& \le  \left[ \int_0^T \int_\Gamma |\partial_{\nu_A} u_1 - \partial_{\nu_A} u_2 |^2 \,\mathrm{d}x \mathrm{d}t \right]^{\frac12} \|\varphi\|_{L^2(\Gamma\times(0,T))},
\end{align*}
which implies
\begin{align*}
\sup_{\|\varphi\|=1}| \mathcal{B}_{ g } (  f_1  - f_2  , \varphi) |
 \le  \| \partial_{\nu_A} u_1 - \partial_{\nu_A} u_2 \|_{L^2(0, T; L^2(\Gamma))}^2 .
\end{align*}
We finish the proof of the theorem by noting the definition of the norm $\|\cdot\|_{\mathcal B}$.
\end{proof}

\section*{Acknowledgments.}
The first author has been supported by ``the Fundamental Research Funds for the Central Universities" (No.B210202147). The second author thanks to the NSF grant of China (No. 11801326). Zhiyuan Li is grateful to Professor Ting Wei for explaining her work on the fractional diffusion-wave equation and for pointing out some typos and misunderstanding in an initial version of this manuscript. Zhiyuan Li also thanks Professor Gongsheng Li for calling his attention to \cite{LY06}.

\end{document}